\documentclass{article}
\usepackage{bm,psfrag,url}
\usepackage{epsfig,psfrag,harvard,bm,dsfont,url}
\usepackage{amsmath,amssymb,amsthm,amsfonts}
\usepackage{amsmath,amssymb,amsfonts}
\usepackage{graphicx,subfigure}

\citationmode{abbr}

\newtheorem{theorem}{Theorem}
\newcommand{\argmax}{\mathop{\mathrm{argmax}}}
\newcommand{\argmin}{\mathop{\mathrm{argmin}}}

\def\by{\mathbf{y}}
\def\bx{\mathbf{x}}
\def\bX{\mathbf{X}}
\def\bA{\mathbf{A}}

\def\bD{\mathbf{D}}
\def\bI{\mathbf{I}}
\def\br{\mathbf{r}}
\def\bs{\mathbf{s}}

\def\bp{\mathbf{p}}
\def\bu{\mathbf{u}}

\def\bb{\mathbf{b}}

\def\bbeta{{\bm\beta}}
\def\btheta{{\bm\theta}}

\def\cA{\mathcal{A}}
\def\cT{\mathcal{T}}
\def\half{\frac{1}{2}}
\def\hbeta{\hat{\beta}}
\def\hu{\hat{u}}
\def\hbbeta{\hat{\bbeta}}
\def\hbtheta{\hat{\btheta}}
\def\hbu{\hat{\bu}}
\def\tmax{\mathrm{max}}
\def\sign{\mathrm{sign}}

\begin{document}
\title{Strong Rules for Discarding Predictors in Lasso-type Problems} 
\author{Robert Tibshirani,
  \footnote{ Departments of Statistics and  Health Research and Policy,
        Stanford University,
        Stanford CA 94305,
        USA. E-mail: tibs@stanford.edu}                                    
\and
Jacob Bien,                                 
\and
Jerome Friedman,
\and
Trevor Hastie,         
\and
Noah Simon,                                 
\and
Jonathan Taylor,
\and
Ryan Tibshirani
}

\maketitle

\begin{abstract}
We consider rules for discarding predictors in lasso regression
and related problems, for computational efficiency.
\citeasnoun{safe} propose ``SAFE'' rules, based on univariate inner products
between each predictor and the  outcome,  that guarantee  
a coefficient will be zero in the solution vector. This provides  a
reduction in the number of variables that need to be entered into the optimization.  
 In this paper, we propose 
{\it strong rules} that are not foolproof but rarely fail in
practice. These are very simple, and can be complemented with
simple checks of the Karush-Kuhn-Tucker (KKT) conditions to ensure  
that the exact solution to the convex problem is delivered. These rules
offer a substantial savings in both computational time and memory,
for a variety of statistical optimization problems. 
\end{abstract}

\section{Introduction}
\label{sec:intro}
Our focus here is on statistical models 
fit using $\ell_1$
regularization. We start with  penalized linear regression.
Consider a problem with $N$
observations and $p$ predictors,
and let $\by$ denote the $N$-vector of outcomes,
and $\bX$ be the $N\times p$ matrix of predictors, with
$j$th column $\bx_j$ and $i$th row $x_i$. For a set of indices 
$\cA = \{j_1,\ldots j_k\}$, we write $\bX_{\cA}$ to denote
the $N \times k$ submatrix $\bX_\cA = [\bx_{j_1},\ldots \bx_{j_k}]$,
and also $\bb_\cA = (b_{j_1},\ldots b_{j_k})$ for a vector $\bb$.
We assume that the predictors and outcome are centered, so we can 
omit an intercept term from the model.

The lasso \citeasnoun{lasso} solves the optimization problem
\begin{equation}
\label{eq:lasso}
\hbbeta = \argmin_{\bbeta} \; 
\half \|\by-\bX\bbeta\|_2^2 +\lambda\|\bbeta\|_1,
\end{equation}
where $\lambda \geq 0$ is a tuning parameter.
There has been considerable work in the past few years 
deriving fast algorithms for this problem, especially for
large values of $N$ and $p$.
A main reason for using the lasso is that the $\ell_1$ penalty 
tends to give exact zeros in $\hbbeta$, and therefore it performs a
kind of variable selection. Now suppose we knew, a priori to
solving \eqref{eq:lasso}, that a subset of the variables $S \subseteq 
\{1,\ldots p\}$ will have zero coefficients in the solution, that is,
$\hbbeta_S = 0$.  
Then we could solve problem \eqref{eq:lasso} with the design matrix
replaced by $\bX_{S^c}$, where $S^c = \{1,\ldots p\} \setminus S $,
for the remaining coefficients $\hbbeta_{S^c}$. 
If $S$ is relatively large, then this could result in a 
substantial computational savings.  
   
\citeasnoun{safe} construct such a set $S$ of ``screened'' or
``discarded'' variables by looking at the inner products 
$|\bx_j^T \by|$, $j=1,\ldots p$. The authors use a clever argument to
derive a surprising set of rules called ``SAFE'', and show that
applying these rules can reduce both time and memory in the overall
computation. In a related work, \citeasnoun{Wu2009} study $\ell_1$
penalized logistic regression and build a screened set $S$ based on
similar inner products.  
However, their construction does not guarantee that the variables in 
$S$ actually have zero coefficients in the solution, and so after
fitting on $\bX_{S^c}$, the authors check the Karush-Kuhn-Tucker (KKT)
optimality conditions for violations. In the case of violations, they
weaken their set $S$, and repeat this process. 
Also, \citeasnoun{FL2008} 
study the screening of variables based on their inner products
in the lasso and related problems, but not from a optimization point
of view. Their screening rules may again set coefficients to zero that
are nonzero in the solution, however, the authors argue that under
certain situations this can lead to better performance in terms of
estimation risk. 

In this paper, we propose {\it strong rules} for discarding predictors
in the lasso and other problems that involve lasso-type penalties. 
These rules discard many more variables than the SAFE rules, but 
are not foolproof, because they can sometimes exclude variables from the model that
have nonzero coefficients in the solution. Therefore
 we rely on KKT conditions to ensure
that we are indeed computing the correct coefficients in the end. 
Our method is most effective for solving problems over a
grid of $\lambda$ values, because we can apply our strong rules
sequentially down the path, which results in a considerable reduction in 
computational time. Generally speaking, the power of the proposed
rules stems from the fact that: 
\begin{itemize}
\item the set of discarded variables $S$ tends to be large and violations
  rarely occur in practice, and
\item the rules are very simple and can be applied to
  many different problems, including the elastic net, lasso penalized
  logistic regression, and the graphical lasso. 
\end{itemize}
In fact, the violations of the proposed rules are so rare, that for a while
a group of us were trying to  establish that they were  foolproof.
At the same time, others in our group were looking for counter-examples
[hence the large number of co-authors!].
After  many flawed proofs, we finally found some counter-examples
to the strong sequential bound (although not to the basic global bound). Despite this, the strong sequential
bound turns out to be extremely useful in practice.

Here is the layout of this paper.
In Section \ref{sec:safe} we review the SAFE rules of
\citeasnoun{safe} for the lasso. The strong rules are introduced and
illustrated in Section \ref{sec:strong} for this same problem. 
We demonstrate 
that the strong rules rarely make mistakes in practice, especially
when $p \gg N$. 
In Section \ref{sec:viol}
we give a condition under which the strong rules do not
erroneously discard predictors (and hence the KKT conditions do not
need to be checked). 
 We discuss the elastic net and penalized logistic regression in
Sections  \ref{sec:en} and \ref{sec:logistic}. Strong rules for more
general convex optimization problems are given in Section
\ref{sec:general}, and these are applied to the graphical lasso.
In Section \ref{sec:glmnet}
 we  discuss how the strong sequential  rule can be used to 
speed up the solution of  convex optimization problems, while still
delivering the exact answer.  We also cover implementation details of the strong sequential rule 
in our {\tt glmnet} algorithm (coordinate descent for lasso penalized 
generalized linear models).
Section \ref{sec:discussion} contains some final discussion.

\section{Review of the SAFE rules}
\label{sec:safe}
The basic SAFE rule of \citeasnoun{safe} for the lasso is defined as
follows: fitting at $\lambda$, we discard predictor $j$ if
\begin{equation}
\label{eq:safe}
|\bx_j^T\by| < \lambda - \|\bx_j\|_2 \|\by\|_2 
\frac{\lambda_\tmax-\lambda}{\lambda_\tmax},
\end{equation}
where $\lambda_\tmax= \max_i |\bx_i^T\by|$ is the smallest $\lambda$ 
for which all coefficients are zero.
The authors derive this bound by looking at a dual of the
lasso problem \eqref{eq:lasso}. This is:
\begin{align}
\label{eq:dual}
\hbtheta = &\argmax_\btheta \; G(\btheta) = \half\|\by\|_2^2 - 
\half\|\by+\btheta\|_2^2 \\
\nonumber
& \text{subject to} \;\; |\bx_j^T \btheta| \leq \lambda \;\; 
\text{for} \;
j=1,\ldots p. 
\end{align}
The relationship between the primal and dual solutions is 
$\hbtheta = \bX \hbbeta - \by$, and 
\begin{equation}
\label{eq:pd}
\bx_j^T \hbtheta \in \begin{cases}
\{+\lambda\} & \text{if}\;\; \hbeta_j > 0 \\
\{-\lambda\} & \text{if}\;\; \hbeta_j < 0 \\
[-\lambda,\lambda] & \text{if}\;\; \hbeta_j = 0
\end{cases}
\end{equation}
for each $j=1,\ldots p$. Here is a sketch of the argument:
first we find a dual feasible point of the form $\btheta_0=s\by$,
($s$ is a scalar), and hence $\gamma=G(s\by)$ represents a lower 
bound for the value of $G$ at the solution. Therefore we can add the
constraint $G(\btheta)\geq \gamma$ to the dual problem \eqref{eq:dual}
and nothing will be changed. For each predictor $j$, we then find 
\begin{equation*}
m_j = \argmax_{\btheta} \;
|\bx_j^T\btheta| \;\; \text{subject to} \; G(\btheta)\geq \gamma. 
\end{equation*}
If $m_j < \lambda$ (note the strict inequality), 
then certainly at the solution $|\bx_j^T \hbtheta| < \lambda$, 
which implies that $\hbeta_j=0$ by \eqref{eq:pd}.  Finally, noting that 
$s=\lambda/\lambda_\tmax$ produces a dual feasible point and rewriting the condition $m_j < \lambda$
gives the rule \eqref{eq:safe}. 

In addition to the basic SAFE bound, the authors also derive a more
complicated but somewhat better bound that they call ``recursive
SAFE'' (RECSAFE). 
As we will show,
the SAFE
rules have the advantage that they will never discard a
predictor when its coefficient is truly nonzero. However, they discard
far fewer predictors than the strong sequential rule, introduced in the
next section.

\section{Strong screening rules}
\label{sec:strong}
\subsection{Basic and strong sequential rules}
Our basic (or global) {\it strong rule} for the lasso problem (\ref{eq:lasso}) discards predictor $j$ 
if  
\begin{equation}
\label{eq:strong}
|\bx_j^T \by| < 2\lambda - \lambda_\tmax,
\end{equation} 
where as before $\lambda_\tmax = \max_j |\bx_j^T \by|$.

 When the
predictors are standardized ($\|\bx_j\|_2=1$ for each $j$), it is not
difficult to see that the right hand side of \eqref{eq:safe} is always
smaller than the right hand side of \eqref{eq:strong}, so that in this
case the SAFE rule is always weaker than the basic strong rule. This follows 
since $\lambda_{\tmax} \leq \|\by\|_2$, so that  
\begin{equation*}
\lambda - \|\by\|_2 
\frac{\lambda_\tmax-\lambda}{\lambda_\tmax}
\; \leq \; \lambda - (\lambda_\tmax-\lambda)
\; = \; 2\lambda-\lambda_\tmax.
\end{equation*}
 Figure \ref{fig:dual} illustrates the SAFE and basic strong rules in an example.
\begin{figure}
\begin{center}
\begin{psfrags}
\psfrag{lm}{$\lambda_{max}$}
\psfrag{la}{$\lambda$}
\psfrag{lambda}{$\lambda$}
\includegraphics[width=3.0in]{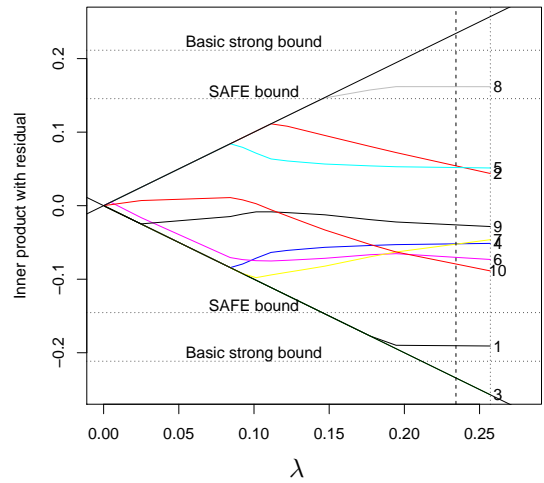}
\end{psfrags}
\end{center}
\caption[fig:dual]{\em SAFE and basic strong  bounds in an example
with 10 predictors, labelled at the right. The plot shows the inner product of each predictor
with the current residual, with the predictors in the model having
maximal inner product equal to $\pm\lambda$.   The dotted vertical line is drawn at $\lambda_{max}$;
the broken vertical line is drawn at $\lambda$. The  strong rule keeps only
predictor \#3, while the SAFE bound keeps  predictors \#8 and \#1 as well.}
\label{fig:dual}
\end{figure}

When the predictors are not standardized, the ordering between the two
bounds is not as clear, but the strong rule still tends to discard more
variables in practice unless the predictors have wildly different
marginal variances.

While  (\ref{eq:strong}) is somewhat useful, its sequential version
is much more powerful.
Suppose that we have already computed the solution
$\hbbeta(\lambda_0)$ at $\lambda_0$, and wish to discard predictors
for a fit at $\lambda<\lambda_0$. Defining the residual
$\br=\by-\bX\hbbeta(\lambda_0)$,  our {\it strong sequential
  rule} discards predictor $j$ if
\begin{equation}
\label{eq:seqstrong}
|\bx_j^T \br| < 2\lambda-\lambda_0.
\end{equation}
Before giving a detailed motivation for these rules,
we first demonstrate their utility.
Figure \ref{fig:fig1} shows some  examples of the
applications of the SAFE and strong rules. There are four scenarios with
various values of $N$ and $p$; in the first three panels,  the $\bX$ matrix is dense,
while it is sparse in the bottom right panel. The population correlation among the
feature is  zero, positive, negative and zero in the four panels.
Finally, 25\% of the coefficients are non-zero, with a standard Gaussian distribution.
In the plots, we are fitting along a path of decreasing $\lambda$ values
and the plots show the number of predictors left after screening
at each stage.
We see that the SAFE and RECSAFE rules only
exclude predictors near the beginning of the path.
The strong rules are  more effective: remarkably,  the strong sequential rule discarded almost all of the predictors that have coefficients of zero.
There were no violations of any of rules in any of the four scenarios.
\begin{figure}
\begin{center}
\includegraphics[width=5.5in]{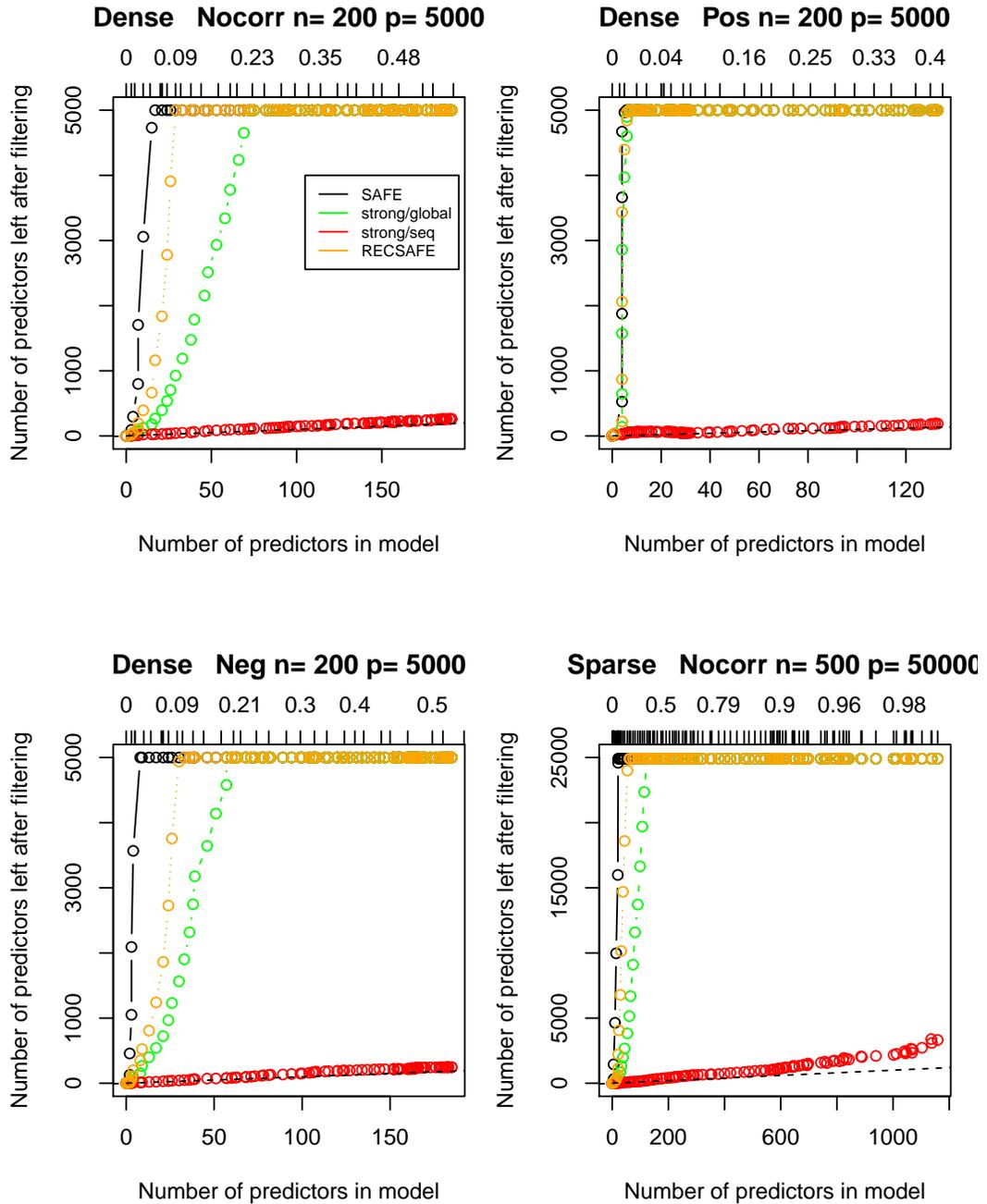}
\end{center}
\caption[fig:fig1]{\em Lasso regression: results of  different  rules
applied to four different scenarios.
There are four scenarios with
various values of $N$ and $p$; in the first three panels  the $\bX$ matrix is dense,
while it is sparse in the bottom right panel. The population correlation among the
feature is  zero, positive, negative and zero in the four panels.
Finally, 25\% of the coefficients are non-zero, with a standard Gaussian distribution.
In the plots, we are fitting along a path of decreasing $\lambda$ values
and the plots show the number of predictors left after screening
at each stage. A broken line with unit slope is added for reference.   The proportion of
variance
explained by the model is  shown along the top of the plot. 
There were no violations of any of the rules in any of the four scenarios.}
\label{fig:fig1}
\end{figure}

It is common practice to standardize the predictors before applying the
lasso, so that the penalty term makes sense.
This is what was done in the examples of Figure \ref{fig:fig1}.
But in some instances, one might not want to standardize the predictors,
and so in  Figure \ref{fig:fig11} we investigate the performance of the rules in this case.
In the left panel 
the population variance of each predictor is the same;
in the right panel it varies by a factor of 50. We see that in the latter case
the SAFE rules  outperform the basic strong rule, but the sequential
strong rule is still the clear winner.
 There were no violations in any of rules in either panel.
\begin{figure}
\begin{center}
\includegraphics[width=4.75in]{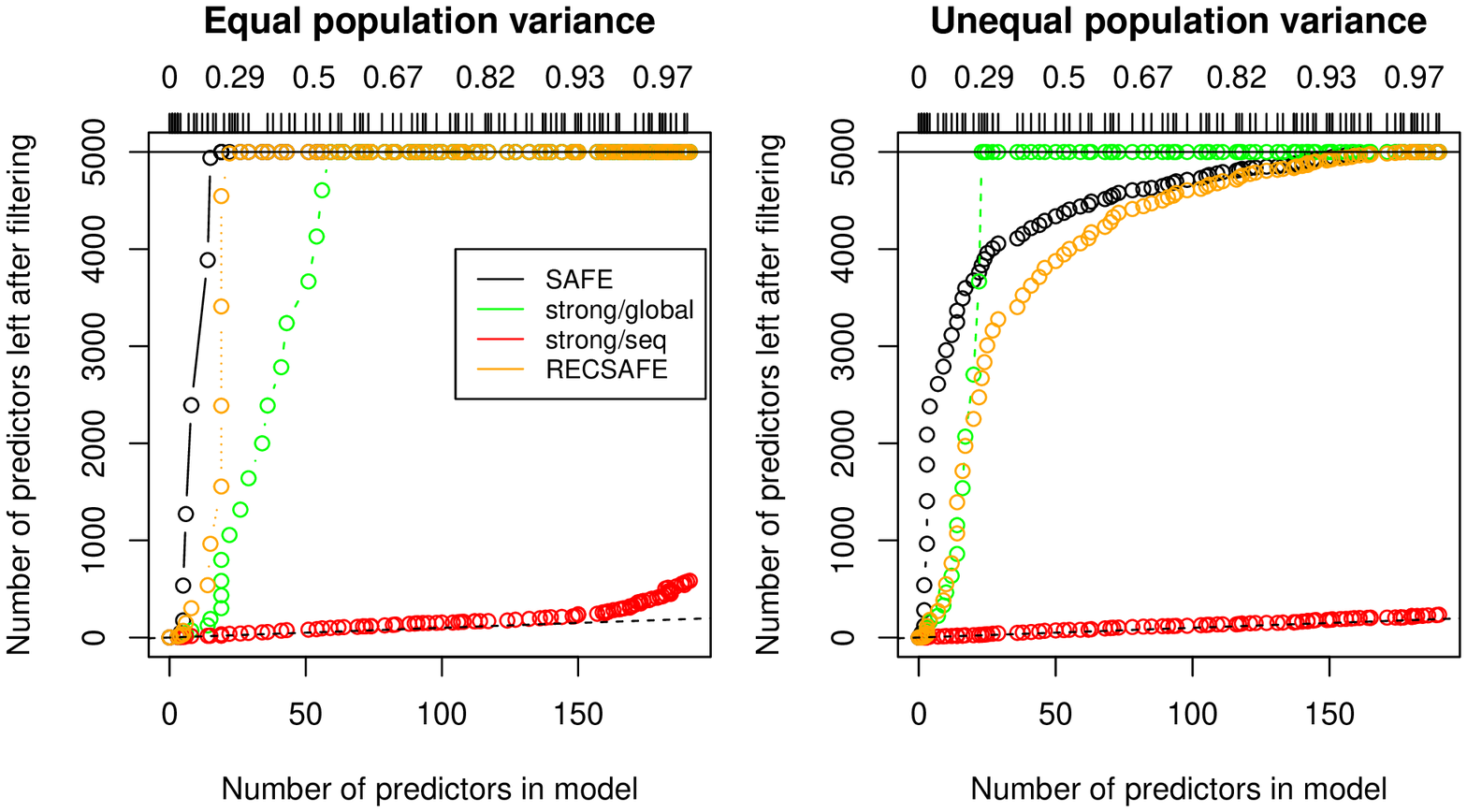}
\end{center}
\caption[fig:fig11]{\em Lasso regression: results of  different  rules
when the predictors are not standardized.
The scenario in the left panel is the same as in the top left panel of Figure \ref{fig:fig1}, except that the features are not standardized before fitting the lasso.
In the data generation for the right panel, each feature is scaled by a random factor between
1 and 50, and again, no standardization is done.}
\label{fig:fig11}
\end{figure}

\subsection{Motivation for the strong rules}
We now give some motivation for the strong rule (\ref{eq:strong})
and later, the sequential rule
 (\ref{eq:seqstrong}).
We start with the KKT
conditions for the lasso problem \eqref{eq:lasso}. These are
\begin{equation}
\label{eq:kkt}
\bx_j^T (\by - \bX\hbbeta) = \lambda \cdot s_j
\end{equation} 
for $j=1,\ldots p$, where $s_j$ is a subgradient of
$\hbeta_j$:  
\begin{equation}
\label{eq:sg}
s_j \in \begin{cases}
\{+1\} & \text{if}\;\; \hbeta_j > 0 \\
\{-1\} & \text{if}\;\; \hbeta_j  < 0 \\
[-1,1] & \text{if}\;\; \hbeta_j  = 0.
\end{cases}
\end{equation}
Let $c_j(\lambda) = \bx_j^T \{\by - \bX\hbbeta(\lambda)\}$, 
where we emphasize the dependence on $\lambda$. Suppose in general
that we could assume
\begin{equation}
\label{eq:slopebound}
|c_j'(\lambda)| \leq 1,
\end{equation}
where $c_j'$ is the derivative with respect to
$\lambda$, and we ignore possible points of
non-differentiability. This would allow us to conclude that
\begin{align}
\label{eq:int1}
|c_j(\lambda_\tmax)-c_j(\lambda)|
&= \left| \int_\lambda^{\lambda_\tmax} 
c_j'(\lambda) \, d\lambda \right| \\
\label{eq:int2}
&\leq \int_\lambda^{\lambda_\tmax} 
|c_j'(\lambda)| \, d\lambda \\
\nonumber
&\leq \lambda_\tmax-\lambda,
\end{align}
and so
\begin{equation*}
|c_j(\lambda_\tmax)| < 2\lambda -\lambda_\tmax 
\;\Rightarrow\;
|c_j(\lambda)| < \lambda 
\;\Rightarrow\;
\hbeta_j(\lambda) = 0,
\end{equation*}
the last implication following from the KKT conditions, \eqref{eq:kkt}
and \eqref{eq:sg}. Then the strong rule \eqref{eq:strong} follows as
$\hbbeta(\lambda_\tmax) = 0$, 
so that $|c_j(\lambda_\tmax)| = |\bx_j^T \by|$. 

Where does the slope condition \eqref{eq:slopebound} come from? 
The product rule applied to
\eqref{eq:kkt} gives
\begin{equation}
\label{eq:pr} 
c_j'(\lambda) = s_j(\lambda) +
\lambda \cdot s_j'(\lambda),
\end{equation}
and as $|s_j(\lambda)| \leq 1$, condition \eqref{eq:slopebound} can be
obtained if we simply drop the second term above. For an active
variable, that is $\hbeta_j(\lambda)\not=0$, we have 
$s_j(\lambda) = \sign\{\hbeta_j(\lambda)\}$, and 
continuity of $\hbeta_j(\lambda)$ with respect to $\lambda$ implies
$s_j'(\lambda) = 0$. But $s_j'(\lambda) \not= 0$ for inactive
variables, and hence the bound \eqref{eq:slopebound} can fail, which
makes the strong rule \eqref{eq:strong} imperfect. It is from this
point of view---writing out the KKT conditions, taking a derivative
with respect to $\lambda$, and dropping a term---that we derive
strong rules for $\ell_1$ penalized logistic regression and more
general problems. 

In the lasso case, condition \eqref{eq:slopebound} has a  more
concrete interpretation. From \citeasnoun{lars}, we know that each
coordinate of the solution $\hbeta_j(\lambda)$ is a
piecewise linear function of $\lambda$, hence so is each inner
product $c_j(\lambda)$. Therefore $c_j(\lambda)$ is differentiable at
any $\lambda$ that is not a kink, the points at which variables enter
or leave the model. In between kinks, condition \eqref{eq:slopebound}
is really just a bound on the slope of $c_j(\lambda)$. The idea 
is that if we assume the absolute slope of $c_j(\lambda)$ is at most
1, then we can bound the amount that $c_j(\lambda)$ changes as we move 
from $\lambda_\tmax$ to a value $\lambda$. Hence if the initial
inner product $c_j(\lambda_\tmax)$ starts too far from the maximal
achieved inner product, then it cannot ``catch up'' in time. An
illustration is given in Figure \ref{fig:strong}. 

\begin{figure}[htb]
\begin{center}
\begin{psfrags}
\psfrag{l}{$\lambda$}
\psfrag{cjxxxxxxxxxx}{$c_j = \bx_j^T (\by - \bX \hbbeta)$}
\psfrag{l0-l1}{$\lambda_\tmax \hspace{-2pt}-\hspace{-2pt} \lambda_1$}
\psfrag{l1}{$\lambda_1$}
\psfrag{l0}{$\lambda_\tmax$}
\includegraphics[width=4in]{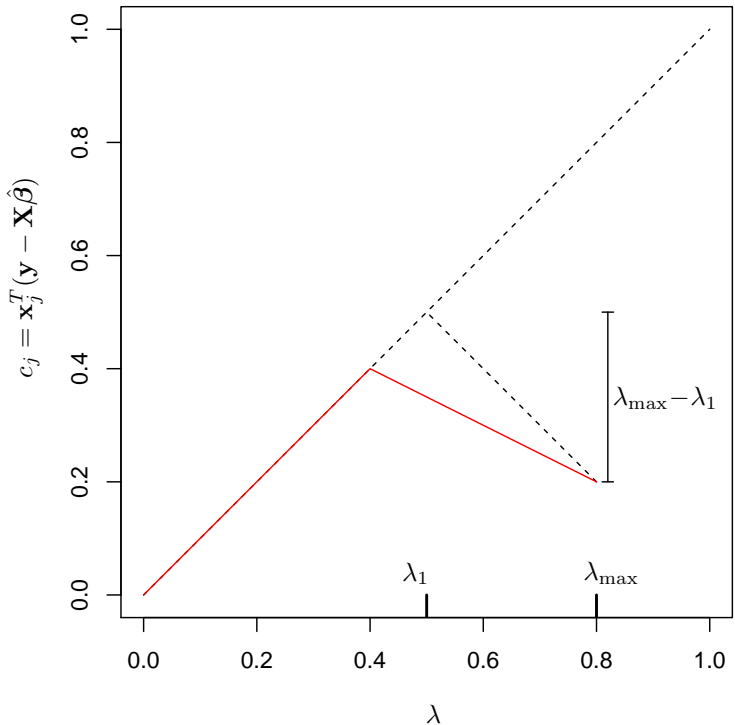}
\end{psfrags}
\caption[fig:strong]{\it Illustration of the slope bound
\eqref{eq:slopebound} leading to the strong rule
\eqref{eq:seqstrong}. 
The inner product $c_j$ is plotted in red 
as a function of $\lambda$, restricted to only one predictor for 
simplicity. The slope of $c_j$ between $\lambda_\tmax$ and 
$\lambda_1$ is bounded in absolute value by 1, so the most it can rise 
over this interval is $\lambda_\tmax-\lambda_1$. Therefore, if it
starts below 
$\lambda_1-(\lambda_\tmax-\lambda_1) = 2\lambda_1-\lambda_\tmax$, it
can not possibly reach the critical level by $\lambda_1$.}
\label{fig:strong}
\end{center}
\end{figure}

The argument for the strong bound (intuitively, an argument about
slopes), uses only local information and so it can be applied to
solving \eqref{eq:lasso} on a grid of $\lambda$ values. 
Hence by the same argument as before, the slope assumption
\eqref{eq:slopebound} leads to the strong sequential rule \eqref{eq:seqstrong}.

It is interesting to note that  
\begin{eqnarray}
|\bx_j^T\br|< \lambda
\label{condKKT}
\end{eqnarray}
is just the KKT condition for excluding a variable in the solution at $\lambda$. 
The strong sequential bound is $\lambda-(\lambda_{0}-\lambda)$
and we can think of the extra term $\lambda_{0}-\lambda$
as a buffer to account for the fact that $|\bx_j^T\br|$
may increase as we move from $\lambda_0$ to $\lambda$.
Note also that as $\lambda_0\rightarrow\lambda$, the  strong sequential rule
becomes the KKT condition (\ref{condKKT}), so that
in effect the sequential rule  at $\lambda_0$  ``anticipates'' the KKT conditions
at $\lambda$.

In summary, it turns out  that the key slope condition (\ref{eq:slopebound}) 
very often holds,
 but can be violated for short stretches, especially when $p\approx N$ and
 for small values
of $\lambda$ in the ``overfit'' regime of a lasso problem.
 In the next section  we provide an example
that shows a violation of the slope bound \eqref{eq:slopebound}, which
breaks the strong sequential rule \eqref{eq:seqstrong}. We 
also give a condition on the design matrix
$\bX$ under which the bound \eqref{eq:slopebound} is guaranteed to
hold. However in simulations in that section,
we find  that these violations are rare in practice
and virtually non-existent when $p>>N$.

\section{Some analysis of the strong rules}
\label{sec:viol}
\subsection{Violation of the slope condition}

Here we demonstrate a counter-example of both the slope bound
\eqref{eq:slopebound} and of the strong sequential rule
\eqref{eq:seqstrong}. We believe that a counter-example for the basic 
strong rule \eqref{eq:strong} can also be constructed, but we have not
yet found one. Such an example is  somewhat more difficult to
construct because it would require that the average slope 
exceed 1 from $\lambda_\tmax$ to $\lambda$, rather than
exceeding 1 for short stretches of $\lambda$ values.  

We took $N=50$ and $p=30$, with the entries of $\by$
and $\bX$ drawn independently from a standard normal
distribution. Then we centered $\by$ and the columns of $\bX$, and 
standardized the columns of $\bX$. As Figure \ref{fig:boundbreak}
shows, the slope of
$c_j(\lambda) =\bx_j^T\{\by-\bX\hbbeta(\lambda)\}$ 
is $c_j'(\lambda)=-1.586$ for all 
$\lambda \in [\lambda_1,\lambda_0]$,
where $\lambda_1=0.0244$, $\lambda_0=0.0259$, and $j=2$.  
Moreover, if we were to use the solution at
$\lambda_0$ to eliminate predictors for the fit at $\lambda_1$, then we
would eliminate the 2nd predictor based on the bound
\eqref{eq:seqstrong}. But this is clearly a problem, because the 2nd
predictor enters the model at $\lambda_1$. By continuity, we can
choose $\lambda_1$ in an interval around $0.0244$ and $\lambda_0$ in
an interval around $0.0259$, and still break the strong sequential
rule \eqref{eq:seqstrong}.  

\begin{figure}[p]
\begin{center}
\begin{psfrags}
\psfrag{l}{$\lambda$}
\psfrag{cjxxxxxxxxxx}{$c_j = \bx_j^T (\by - \bX \hbbeta)$}
\psfrag{bound}{$2\lambda_1-\lambda_0$} 
\psfrag{l1}{$\lambda_1$}
\psfrag{l0}{$\lambda_0$}
\includegraphics[width=5.75in]{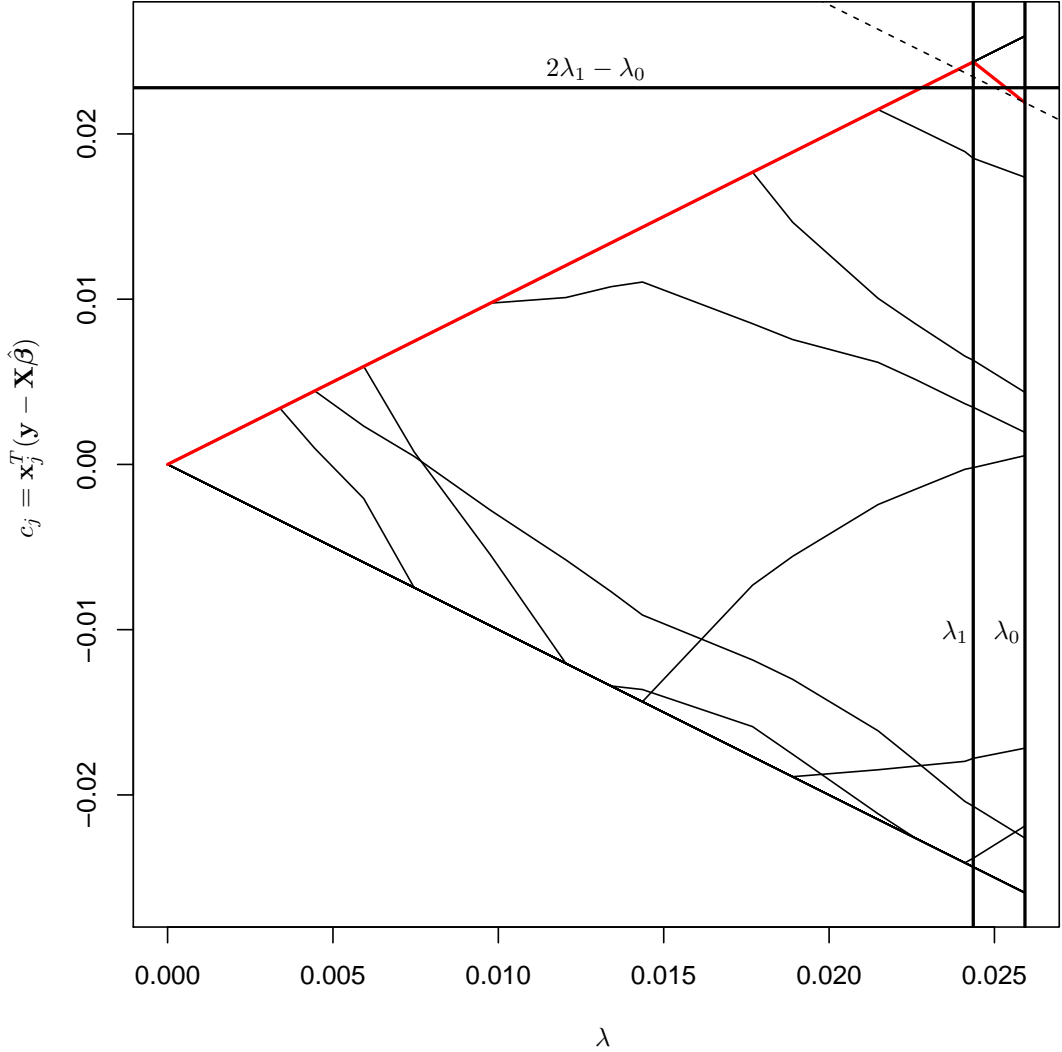}
\end{psfrags}
\caption[fig:boundbreak]{\it Example of a violation of
the slope bound \eqref{eq:slopebound}, which breaks the strong
sequential rule \eqref{eq:seqstrong}. The entries of $\by$ and $\bX$
were generated as independent, standard normal random variables with  
$N=50$ and $p=30$. (Hence there is no underlying signal.) 
The  lines with slopes $\pm \lambda$  are the envelop of maximal inner products 
achieved by predictors in the model for each $\lambda$.
For clarity we only show a short stretch of the solution path. 
The rightmost vertical line is drawn at $\lambda_0$, and we are
considering the new value $\lambda_1<\lambda_0$, the vertical line to
its left. The horizontal line is the bound \eqref{eq:slopebound}. 
In the top right part of the plot, the inner product path for the
predictor $j=2$ is drawn in red, and starts below the bound, but enters
the model at $\lambda_1$. The slope of the red segment between
$\lambda_0$ and $\lambda_1$ exceeds 1 in absolute value. A dotted line
with slope -1 is drawn beside the red segment for reference.}
\label{fig:boundbreak}
\end{center}
\end{figure}

\subsection{A sufficient condition for the slope bound}
\label{sec:dd}
\citeasnoun{genlasso} prove a general result that can be used to
give the following sufficient condition for the unit slope bound
\eqref{eq:slopebound}. 
Under this condition, both basic and strong sequential rules are guaranteed
not to fail.

Recall that a matrix $\bA$ is diagonally dominant if $|A_{ii}| \geq
\sum_{j\neq i} |A_{ij}|$ for all $i$. Their result gives us the
following: 

\begin{theorem}
Suppose that $\bX$ is $N\times p$, with $N \geq p$, and of full rank. 
If
\begin{equation}
\label{eq:dd}
(\bX^T\bX)^{-1}\;\mathrm{is}\;\mathrm{diagonally}\;\mathrm{dominant},
\end{equation}
then the slope bound \eqref{eq:slopebound} holds at all points
where $c_j(\lambda)$ is differentiable, for $j=1,\ldots p$, and hence
the strong rules \eqref{eq:strong}, \eqref{eq:seqstrong} never produce 
violations. 
\end{theorem}
\begin{proof}
\citeasnoun{genlasso} consider a generalized lasso problem 
\begin{equation}
\label{eq:genlasso}
\argmin_\bbeta \; \half\|\by-\bX\hbbeta\|_2^2 + \lambda\|\bD\bbeta\|_1,
\end{equation}
where $\bD$ is a general $m \times p$ penalty matrix. In the proof of
their ``boundary lemma'', Lemma 1, they show that if
$\mathrm{rank}(X)=p$ and $\bD(\bX^T\bX)^{-1}\bD^T$ is diagonally
dominant, then the dual solution $\hbu(\lambda)$ corresponding to
problem \eqref{eq:genlasso} satisfies
\begin{equation*}
|\hu_j(\lambda)-\hu_j(\lambda_0)| \leq |\lambda-\lambda_0|
\end{equation*}
for any $j=1,\ldots m$ and $\lambda,\lambda_0$. By piecewise linearity
of $\hu_j(\lambda)$, this means that $|\hu'_j(\lambda)| \leq 1$ at all
$\lambda$ except the kink points. Furthermore, when $\bD=\bI$, problem 
\eqref{eq:genlasso} is simply the lasso, and it turns out that the
dual solution $\hu_j(\lambda)$ is exactly the inner product
$c_j(\lambda) = \bx_j^T \{\by-\bX\hbbeta(\lambda)\}$. This proves the
slope bound \eqref{eq:slopebound} under the condition that
$(\bX^T\bX)^{-1}$ is diagonally dominant.

Finally, the kink points are countable and hence form a set of
Lebesgue measure zero. Therefore $c_j(\lambda)$ is differentiable
almost everywhere and the integrals in \eqref{eq:int1} and
\eqref{eq:int2} make sense. This proves the strong rules
\eqref{eq:strong} and \eqref{eq:seqstrong}.  
\end{proof}

We note a similarity between condition \eqref{eq:dd} and the positive
cone condition used in \citeasnoun{lars}. It is not hard to see that 
the positive cone condition implies \eqref{eq:dd}, and actually
\eqref{eq:dd} is a much weaker condition because it doesn't require
looking at every possible subset of columns. 

A simple model in which diagonal dominance holds is when the columns 
of $\bX$ are orthonormal, because then $\bX^T\bX = \bI$. But the
diagonal dominance condition \eqref{eq:dd} certainly holds outside of
the orthogonal design case. We give two such examples below.

\begin{itemize}
\item {\it Equi-correlation model.} Suppose that $\|\bx_j\|_2=1$ for
  all $j$, and $\bx_j^T \bx_k = r$ for all $j\not=k$. Then the inverse
  of $\bX^T \bX$ is 
\begin{equation*}
(\bX^T\bX)^{-1} = \bI\cdot\frac{1}{1-r} -\frac{1}{1-r} \left(
\frac{\mathbf{1}\mathbf{1}^T}{1+r(p-1)} \right),
\end{equation*}
where $\mathbf{1}$ is the vector of all ones. This is diagonally
dominant as along as $r \geq 0$.
\item {\it Haar basis model.} Suppose that the columns of $\bX$ form a
  Haar basis, the simplest example being 
\begin{equation}
\label{eq:lt}
\bX = \left(\begin{array}{cccc}
1 & & & \\
1 & 1 & & \\
\vdots & & & \\
1 & 1 & \ldots & 1
\end{array}\right),
\end{equation}
the lower triangular matrix of ones. Then $(\bX^T \bX)^{-1}$ is
diagonally dominant.
This arises, for example, in the one-dimensional fused lasso where we 
solve
\begin{equation*}
\argmin_\bbeta \; \half \sum_{i=1}^N (y_i-\beta_i)^2 +\lambda
\sum_{i=2}^N |\beta_i-\beta_{i-1}|. 
\end{equation*}
If we transform this problem to the parameters $\alpha_1=1$,
$\alpha_i = \beta_i-\beta_{i-1}$ for $i=2,\ldots N$, then we get a
lasso with design $\bX$ as in \eqref{eq:lt}.
\end{itemize}

\subsection{Connection to the irrepresentable condition}
The slope bound \eqref{eq:slopebound} possesses an interesting
connection to a concept called the ``irrepresentable condition''. 
Let us write $\cA$ as the set of active variables at $\lambda$,
\begin{equation*}
\cA = \{j: \hbeta_j(\lambda) \not= 0\},
\end{equation*} 
and $\|\bb\|_\infty = \max_i |b_i|$ for a vector $\bb$. Then, using
the work of \citeasnoun{lars}, we can express the slope condition
\eqref{eq:slopebound} as 
\begin{equation}
\label{eq:slopebound2}
\|\bX_{\cA^c}^T \bX_\cA (\bX_\cA^T \bX_\cA)^{-1} \sign(\hbbeta_\cA)
\|_\infty \leq 1, 
\end{equation}
where by $\bX_\cA^T$ and $\bX_{\cA^c}^T$, we really mean $(\bX_\cA)^T$
and $(\bX_{\cA^c})^T$, and the sign is applied element-wise. 

On the other hand, a common condition appearing in work about model
selection properties of lasso, in both the finite-sample and
asymptotic settings, is the so called ``irrepresentable condition''
\citeasnoun{lassomodel,sharpthresh,nearideal}, which is closely related to
the concept of ``mutual incoherence'' \citeasnoun{fuchs,tropp,lassograph}. 
Roughly speaking, if $\bbeta_\cT$ denotes the nonzero coefficients in
the true model, then the irrepresentable condition is that
\begin{equation}
\label{eq:irrcond}
\|\bX_{\cT^c}^T \bX_\cT (\bX_\cT^T \bX_\cT)^{-1} \sign(\bbeta_\cT)
\|_\infty \leq 1-\epsilon
\end{equation}
for some $0 < \epsilon \leq 1$. 

The conditions \eqref{eq:irrcond} and \eqref{eq:slopebound2} appear 
extremely similar, but a key difference between the two is 
that the former pertains to the true coefficients that generated the
data, while the latter pertains to those found by the lasso
optimization problem. Because $\cT$ is
associated with the true model, we can put
a probability distribution on it and a probability distribution on
$\sign(\bbeta_\cT)$, and then show that with high probability, certain
designs $\bX$ are mutually incoherent \eqref{eq:irrcond}. For
example, \citeasnoun{nearideal} 
suppose that  $k$ is sufficiently small, $\cT$ is drawn from the uniform
distribution on $k$-sized subsets of $\{1,\ldots p\}$, and each
entry of $\sign(\bbeta_\cT)$ is equal to $+1$ or $-1$ 
with probability $1/2$, independent of each other. 
Under this model, they show that designs $\bX$
with $\max_{j \not= k} |\bx_j^T \bx_k| = O(1/\log{p})$ satisfy the
irrepresentable condition with very high probability. 

Unfortunately the same types of arguments cannot be applied directly
to \eqref{eq:slopebound2}. A distribution on $\cT$ and
$\sign(\bbeta_\cT)$ induces a different distribution on $\cA$ and
$\sign(\hbbeta_\cA)$, via the lasso optimization procedure. Even if
the distributions of $\cT$ and $\sign(\bbeta_\cT)$ are very simple, the
distributions of $\cA$ and $\sign(\hbbeta_\cA)$ can become quite
complicated. Still, it does not seem hard to believe that confidence
in \eqref{eq:irrcond} translates to some amount of confidence in
\eqref{eq:slopebound2}. Luckily for us, we do not need the slope bound 
\eqref{eq:slopebound2} to hold exactly or with any specified level of
probability, because we are using it as a computational tool and can
simply revert to checking the KKT conditions when it fails.

\subsection{A numerical investigation of the  strong sequential  rule violations}
We generated Gaussian data with $N=100$, varying values of the number of predictors $p$ and pairwise correlation 0.5 between the predictors.
One quarter of the coefficients were non-zero, with the indices of the nonzero  predictors randomly chosen and their
values equal to
$\pm 2$. We  fit the lasso for 80 equally spaced values of $\lambda$ from $\lambda_{max}$ to 0, and
recorded the number of violations of the strong sequential rule.  Figure \ref{fig:numbreaks}
shows the  number of violations (out of $p$ predictors) averaged over 100 simulations: we plot versus the percent variance explained
instead of $\lambda$, since the former is more meaningful.
Since the  vertical axis is the  total  number of violations,
we see that violations are quite rare in general never averaging more than 0.3 out of $p$ predictors. They  are more common
near the right end of the path.
They also tend to occur when $p$ is fairly close to  $N$. When $p\gg N$  ($p=500$ or $1000$ here),
there were no violations. Not surprisingly, then, there were no violations in the numerical examples in this paper
since they all have $p\gg N$.
\begin{figure}
\begin{center}
\includegraphics[width=3.0in]{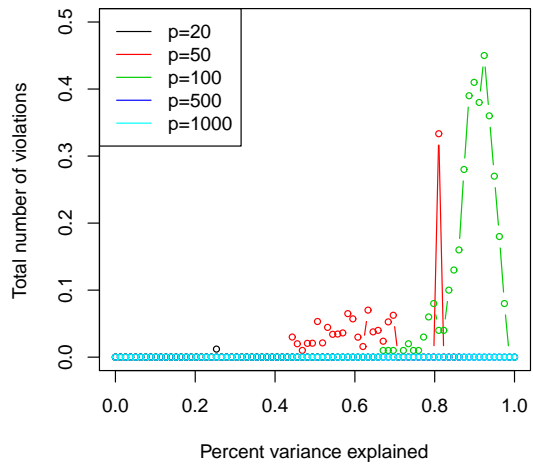}
\end{center}
\caption[fig:numbreaks]{\em Total number of violations (out of $p$ predictors) of the strong sequential rule, for
simulated data with $N=100$ and varying values of $p$. A sequence of models is fit, with decreasing values of
 $\lambda$ as we move from left to right.  The features are uncorrelated.
The results are averages over 100 simulations.}
\label{fig:numbreaks}
\end{figure}

Looking at (\ref{condKKT}),
it suggests that if we take a finer grid of $\lambda$ values, there should be fewer violations of the rule.
However we have not found this to be  true numerically: the average number of violations at each grid point $\lambda$ stays about the same.

\section{Screening rules for the elastic net}
\label{sec:en}
In the elastic net we solve the problem
\footnote{This differs from the original form of the ``naive'' elastic
net in \citeasnoun{ZH2005} by the factors of $1/2$, just for notational convenience.}
\begin{eqnarray}
{\rm minimize}\;\frac{1}{2}||\by-\bX\bbeta||^2 +\frac{1}{2}\lambda_2||\bbeta||^2+\lambda_1  ||\bbeta||_1 
\label{eqn:en}
\end{eqnarray}
Letting 
\begin{eqnarray}
\bX^*=\begin{pmatrix}
\bX\\
\sqrt{\lambda_2}\cdot\bI\\
\end{pmatrix}
;\;\;
\by^*=\begin{pmatrix}
\by\\
0\\
\end{pmatrix},
\end{eqnarray}
we can write  (\ref{eqn:en}) as
\begin{eqnarray}
{\rm minimize}\frac{1}{2}||\by^*-\bX^*\bbeta||^2 +  \lambda_1||\bbeta||_1.
\label{eqn:en2}
\end{eqnarray}
In this form we can apply the SAFE rule (\ref{eq:safe}) to obtain a rule for discarding predictors.
Now $|{\bx_j^*}^T\by^*|= |\bx_j^T\by|$, $||\bx_j^*||=
\sqrt{||\bx_j||^2+\lambda_2}$, $||\by^*||=||\by||$.
Hence  the  global rule for discarding 
predictor $j$ is
\begin{eqnarray}
 |\bx_j^T\by| < \lambda_1 -||\by||\cdot \sqrt{||\bx_j||^2+\lambda_2}\cdot \frac{\lambda_{1max}-\lambda_1}{\lambda_{1max}}
\label{eqn:englobal}
\end{eqnarray}

Note that the {\tt glmnet} package uses the parametrization
$((1-\alpha)\lambda,\alpha\lambda)$ rather than
$(\lambda_2,\lambda_1)$. With this parametrization the basic SAFE rule has the
form
\begin{eqnarray}
 |\bx_j^T\by| < \Bigl(\alpha\lambda -||\by||\cdot \sqrt{||\bx_j||^2+(1-\alpha)\lambda}\cdot \frac{\lambda_{max}-\lambda}{\lambda_{max}}\Bigr)
\label{eqn:glmnetglobal}
\end{eqnarray}

The strong screening rules turn out to be the same as for the lasso.
With the {\tt glmnet} parametrization the global rule is simply
\begin{eqnarray}
 |\bx_j^T\by| < \alpha(2\lambda - \lambda_{max})
\label{eqn:englobala}
\end{eqnarray}
while the sequential rule is
\begin{eqnarray}
 |\bx_j^T\br| < \alpha(2\lambda - \lambda_{0}).
\label{eqn:enseqa}
\end{eqnarray} 
\begin{figure}
\begin{center}
\begin{psfrags}
\psfrag{alpha=  0.1}{$\alpha=0.1$}
\psfrag{alpha=  0.5}{$\alpha=0.5$}
\psfrag{alpha=  0.9}{$\alpha=0.9$}
\includegraphics[width=5.3in]{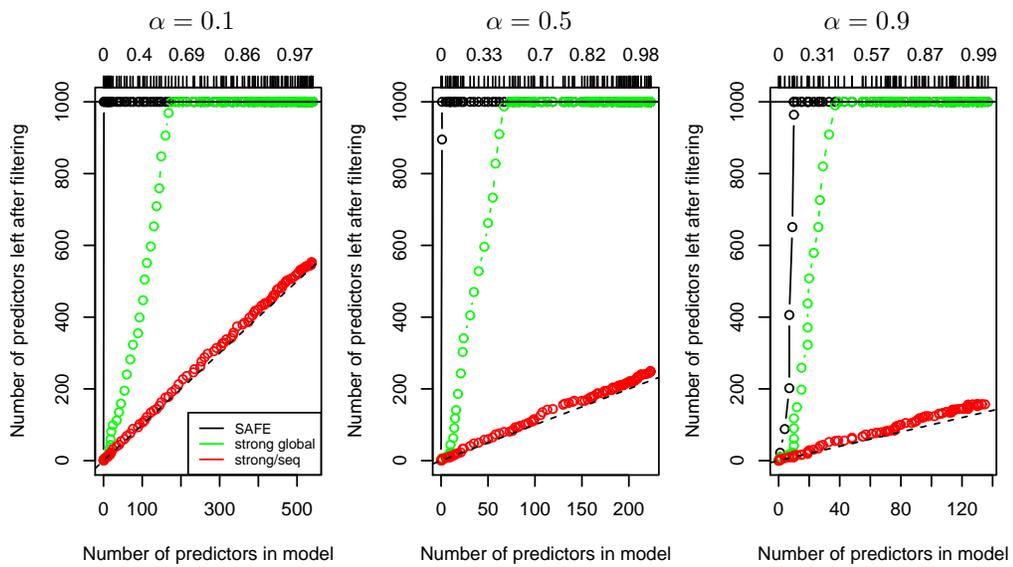}
\end{psfrags}
\end{center}
\caption[fig:en]{\em Elastic net: results for different rules
for three different values of
the mixing parameter $\alpha$.
In the plots, we are fitting along a path of decreasing $\lambda$ values
and the plots show the number of predictors left after screening
at each stage. The proportion of
variance
explained by the model is  shown along the top of the plot is shown.
There were no violations of any of the rules in the 3 scenarios.}
\label{fig:en}
\end{figure}

Figure~\ref{fig:en} show results for the elastic net with standard independent
Gaussian data, $n=100, p=1000$, for 3 values of
$\alpha$. 
 There were no violations in any of these figures,
i.e. no predictor was discarded that had a non-zero coefficient
at the actual solution.  Again we see that the
strong sequential rule performs extremely well, leaving only a small number
of excess
predictors at each stage.

\section{Screening rules for logistic regression}
\label{sec:logistic}
Here we have a binary response $y_i=0,1$ and
we assume the  logistic model
\begin{eqnarray}
{\rm Pr}(Y=1|x)=1/(1+\exp(-\beta_0-x^T\beta))
\end{eqnarray}
Letting $p_i={\rm Pr}(Y=1|x_i)$, the
 penalized log-likelihood is
\begin{eqnarray}
\ell(\beta_0,\bbeta)=-\sum_i[ y_i \log p_i +(1-y_i) \log(1-p_i)]+\lambda||\beta||_1
\end{eqnarray}

\citeasnoun{safe} derive an exact global rule for discarding predictors,
based on the inner products between $\by$ and each predictor,
using the same kind of dual argument as in the Gaussian case.

Here we investigate the analogue of the strong rules
(\ref{eq:strong}) and (\ref{eq:seqstrong}).
The subgradient equation for logistic regression is
\begin{eqnarray}
\bX^T(\by-\bp(\bbeta))=\lambda\cdot{\rm sign}(\bbeta)
\end{eqnarray}

This leads to the global rule: 
letting $\bar\bp={\bf 1}\bar y$, $\lambda_{max}={\rm max}|\bx_j^T(\by-\bar\bp)|$,
we discard  predictor $j$ if
\begin{eqnarray}
|\bx_j^T(\by-\bar\bp)|< 2\lambda-\lambda_{max}
\label{eqn:logit}
\end{eqnarray}
The sequential version, starting at $\lambda_0$, uses 
$\bp_0=\bp(\hat\beta_0(\lambda_0),\hat\bbeta(\lambda_0))$:
\begin{eqnarray}
|\bx_j^T(\by-\bp_0)|< 2\lambda-\lambda_0.
\label{eqn:logitseq}
\end{eqnarray}

\begin{figure}
\begin{center}
\includegraphics[width=3.25in]{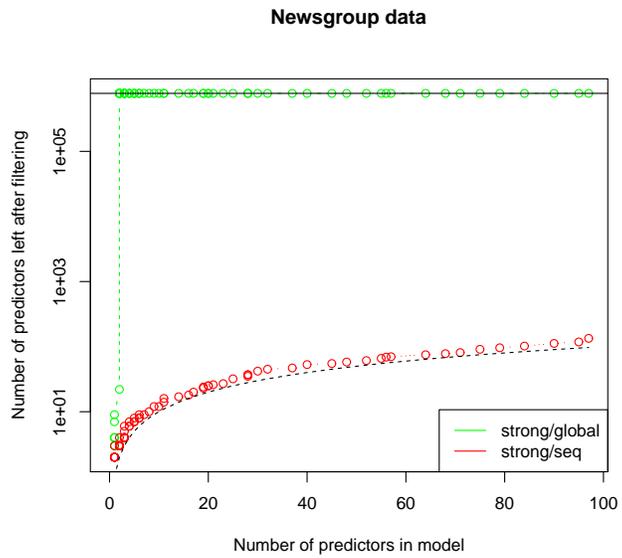}
\end{center}
\caption[fig:newsgroup]{\em Logistic regression: results for newsgroup example, using the new global rule (\ref{eqn:logit})  
and the new sequential rule (\ref{eqn:logitseq}).
The broken black curve is the 45$^o$ line, drawn on the log scale.}
\label{fig:newsgroup}
\end{figure}

Figure \ref{fig:newsgroup} show the result of various rules in an example, the newsgroup  document classification
  problem \cite{lang95:_newsw}. We used the training set cultured from these data by
  \citeasnoun{koh07:_l1}. The response is binary, and indicates a subclass
  of topics; the predictors are binary, and indicate the presence of particular tri-gram
  sequences. The predictor matrix has $0.05\%$ nonzero values.
\footnote{This dataset is available as a saved {R} data object at 
{\tt http://www-stat.stanford.edu/~hastie/glmnet}}
Results for are shown for  the new global rule (\ref{eqn:logit})  
and the new sequential rule (\ref{eqn:logitseq}).
We were unable to compute the logistic regression global  SAFE rule for this example,
using our  R language implementation, as this had a very long  computation time. 
 But in smaller examples it performed much like the global  SAFE rule in the Gaussian case.
Again we see that the strong sequential rule (\ref{eqn:logitseq}), after computing
the inner product of the residuals with all predictors at each stage, allows us to discard
the vast majority of the predictors before fitting. There were no violations of either rule in this example.

Some approaches to penalized logistic regression such as the {\tt glmnet} package
 use a weighted least squares iteration within a Newton step. For these algorithms,
an alternative 
approach to discarding predictors  would be to apply one of the Gaussian rules
within the weighted least squares iteration. 

\citeasnoun{Wu2009} used $|\bx_j^T(\by-\bar\bp)|$ to screen predictors (SNPs) in genome-wide association studies,
where the number of variables can exceed a million. Since they only anticipated models with say $k<15$ terms, they selected a small multiple, say  $10k$,  of SNPs and computed the lasso solution path to $k$ terms. All the screened SNPs could then be checked for violations to verify that the solution found was global.

\section{Strong rules for general problems}
\label{sec:general}
Suppose that we have a convex  problem of the form
\begin{eqnarray}
{\rm minimize}_\bbeta \Big[ f(\bbeta)+\lambda\cdot\sum_{k=1}^K g(\bbeta_j)\Bigr]
\end{eqnarray}
where $f$ and $g$ are convex functions, $f$ is differentiable
 and $\bbeta=(\bbeta_1,\bbeta_2, \ldots \bbeta_K)$ with each $\bbeta_k$ being a scalar or vector.
Suppose further that the subgradient equation for this problem has the form
\begin{eqnarray}
f'(\bbeta)+\lambda\cdot\bs_k=0;\; k=1,2,\ldots K
\end{eqnarray}
where each subgradient variable $\bs_k$ satisfies $||\bs_k||_q\leq A$,
and $||\bs_k||_q=A$ when the constraint $g(\bbeta_j)=0$ is satisfied
(here $||\cdot||_q$ is a norm).
Suppose that we have two values $\lambda<\lambda_0$, and corresponding
solutions $\hat\bbeta(\lambda), \hat\bbeta(\lambda_0)$.
Then following the same logic as in Section \ref{sec:strong}, we 
can derive the general strong  rule
\begin{eqnarray}
||\frac{f(\hat\bbeta_{0k}}{d\bbeta_k})||_q<  (1+A)\lambda- A\lambda_0
\label{genstrong}
\end{eqnarray}
This can be applied either globally or sequentially.
In the lasso regression setting,
it is easy to check that
this reduces to the rules   (\ref{eq:strong}),(\ref{eq:seqstrong})  where $A=1$.

The rule (\ref{genstrong}) has many potential applications.
For example in the graphical lasso for sparse inverse covariance estimation
\cite{FHT2007},
we observe
  $N$ multivariate normal observations of dimension $p$, with mean $0$
and covariance $\Sigma$,
with  observed empirical covariance matrix $S$.
Letting $\Theta=\Sigma^{-1}$,
the problem is to maximize  the penalized log-likelihood
\begin{eqnarray}
\log \det \Theta-{\rm tr}(S\Theta)-\lambda||\Theta||_1,
\label{eqn:gl}
\end{eqnarray}
over non-negative definite matrices  $\Theta$.
The penalty $||\Theta||_1$ sums the absolute values of the entries of $\Theta$;
we assume that the diagonal is not penalized.
The subgradient equation is 
\begin{equation}
\Sigma-S-\lambda\cdot \Gamma = 0,
\label{gradd}
\end{equation}
where  $\Gamma_{ij}\in{\rm
  sign}(\Theta_{ij})$.
One could apply the rule  (\ref{genstrong}) elementwise, and this
would be useful for an optimization method that operates elementwise. 
This gives a rule of the form $|S_{ij}-\hat\Sigma(\lambda_0)|< 2\lambda-\lambda_0$.
However, the graphical lasso algorithm proceeds in a blockwise fashion,
optimizing one whole row and column at a time.
Hence for the graphical lasso, it is more effective to discard entire rows and columns at once.
For each row $i$,
    let $s_{12}$, $\sigma_{12}$, and $\Gamma_{12}$ denote $S_{i,-i}$,
    $\Sigma_{i,-i}$, and $\Gamma_{i,-i}$, respectively.
Then
 the subgradient equation for one 
row has the form
\begin{equation}
\sigma_{12}-s_{12}-\lambda\cdot \Gamma_{12} = 0,
\label{gradd2}
\end{equation}
Now given two values $\lambda<\lambda_0$, and solution
$\hat\Sigma^0$ at $\lambda_0$,
we form the sequential rule
\begin{equation}
{\rm max}|\hat\sigma^0_{12}-s_{12}| <2\lambda-\lambda_0.
\label{strong:gl}
\end{equation}
If this rule is satisfied, we discard the entire $i$th row and column of $\Theta$,
and hence set them to zero (but retain the $i$th diagonal
element).
Figure \ref{glasso} shows an example with $N=100,p=300$, standard independent
Gaussian variates.
No violations of the rule occurred.
\begin{figure}
\begin{center}
\includegraphics[width=3.25in]{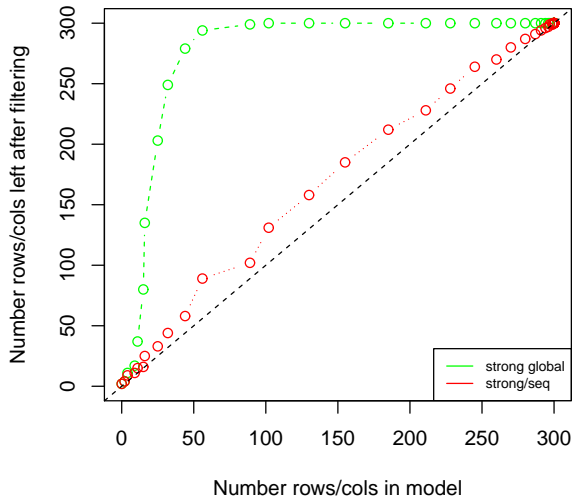}
\end{center}
\caption[glasso]{\em Strong global and sequential rules applied to  the graphical
lasso. A broken line with unit slope is added for reference.}
\label{glasso}
\end{figure}

Finally, we note that
strong rules can be derived in a similar way, for other problems such as the
 {group lasso} \cite{YL2007}.
In particular, if $\mathbf X_\ell$ denotes the
    $n\times p_\ell$ block of the design matrix corresponding to the
    features in the $\ell$th group, then the strong sequential rule is simply
    \begin{align*}
      ||\mathbf X_\ell^T\mathbf r||_2 < 2\lambda-\lambda_{\max}.
    \end{align*}
When this holds, we set $\boldsymbol\beta_\ell=\boldsymbol0$.
%
%

\section{Implementation and numerical studies}
\label{sec:glmnet}

The strong sequential  rule (\ref{eq:seqstrong}) can be used to provide potential speed
improvements in convex optimization problems. Generically, given a solution $\hat\bbeta(\lambda_0)$
and considering a new value $\lambda< \lambda_0$, let $S(\lambda)$ be the indices of the predictors
that survive the screening rule  (\ref{eq:seqstrong}): we call this the {\em strong set}. Denote by $E$ 
the eligible set of predictors.
Then a useful strategy would be
\begin{enumerate}
\item Set $E=S(\lambda)$.
\item Solve the problem at value $\lambda$ using only the predictors in $E$.
\item Check  the KKT conditions at this solution for all predictors. If there are no violations,
we are done. Otherwise add the predictors that violate the KKT conditions to the set $E$,
and repeat steps (b) and (c).
\end{enumerate}
Depending on how the optimization is done in step (b), this can be quite effective.
Now in  the {\tt glmnet} procedure,  coordinate descent  is used, with warm starts
over a grid of decreasing values of $\lambda$.
In addition, an ``ever-active'' set of predictors $A(\lambda)$ is maintained, consisting of the indices of all
predictors that have a non-zero coefficient for some $\lambda'$ greater than the
current value $\lambda$  under consideration. The solution is first found for this active set:
then the KKT conditions are checked for all predictors. 
if there are no violations, then we have the solution at $\lambda$; otherwise we add 
the violators into the active set and repeat.

The two strategies above are very similar, with one using the strong set $S(\lambda)$ and the
other using the ever-active set $A(\lambda)$.
Figure \ref{fig:active} shows the active and strong sets for an example.
\begin{figure}
\begin{center}
\includegraphics[width=5.5in]{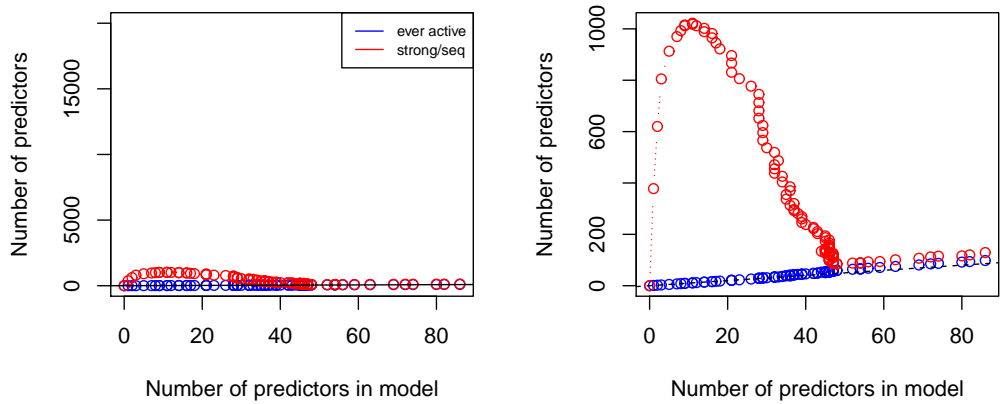}
\end{center}
\caption[fig:active]{\em  Gaussian lasso setting, $N=200, p=20,000$, pairwise correlation  between features of $0.7$.
The first 50 predictors have positive, decreasing coefficients.
Shown are the number of predictors left after applying the strong sequential rule (\ref{eq:seqstrong})
and the number that have ever been active (i.e. had a non-zero coefficient in the solution)
for values of $\lambda$ larger than the current value. A broken line with unit slope is added for reference. The right-hand plot is a zoomed version of the left plot.}
\label{fig:active}
\end{figure}
Although the strong rule greatly reduces the total number of predictors,
it  contains more predictors than the ever-active set;
accordingly, violations occur more often in the ever-active set than the  strong set.
This effect is due to the high correlation between features and the fact that the
signal variables have coefficients of the same sign. It also occurs with logistic regression with lower correlations, say 0.2.

In light of this,  we find that using both $A(\lambda)$ and $S(\lambda)$ can be advantageous.
For {\tt glmnet} we adopt the following combined strategy:
\begin{enumerate}
\item Set $E=A(\lambda)$.
\item Solve the problem at value $\lambda$ using only the predictors in $E$.
\item Check  the KKT conditions at this solution for all predictors in
$S(\lambda)$. If there are  violations, add these predictors into $E$, and go back to
step (a) using the current solution as a  warm start.
\item Check the KKT conditions for all predictors. If there are no violations, we are done.
Otherwise add these violators into $A(\lambda)$, recompute $S(\lambda)$ and go back to step (a) using the 
 current solution as a  warm start.
\end{enumerate}
Note that violations in step (c) are fairly common, while those in step (d) are  rare.
Hence the fact that the size of $S(\lambda)$ is $\ll p$ can make this an effective strategy.

We implemented this strategy and compare it to the standard {\tt glmnet} algorithm
in a variety of problems, shown in Tables
1--3.
Details are given in the table captions.
We see that the  new strategy offers a speedup factor of five or more in some cases,
and never seems to slow things down.

The strong sequential rules also have the potential for space savings.
With a large dataset,  one could compute the inner products $\{\bx_j^T\br\}_1^p$
offline to determine the strong set of predictors, and then
carry out the intensive optimization steps in memory using just this
subset of the predictors.
\section{Discussion}
\label{sec:discussion}
In this paper we have proposed strong global and sequential rules for
discarding predictors in statistical convex optimization problems such as the lasso.
When combined with checks of the KKT conditions,
these  can offer substantial improvements in speed while still yielding the exact solution.
We plan to include these rules in a future version of the  {\tt glmnet} package.

The RECSAFE method uses the solution at a given point $\lambda_0$
to derive a rule for discarding predictors at $\lambda<\lambda_0$.
Here is another way to 
(potentially) apply the SAFE rule in  a sequential manner.
Suppose that we have $\hat\bbeta_0=\hat\bbeta(\lambda_0)$,
and $\br=\by-\bX\hat\bbeta_0$,
and we consider the fit at
 $\lambda<\lambda_0$, with
$\br=\by-\bX\hat\bbeta_0$.
Defining
\begin{eqnarray}
\lambda_{0}&=&{\rm max}_j(|\bx_j^T\br|);
\end{eqnarray}
we  discard predictor $j$ if
\begin{eqnarray}
|\bx_j^T\br|< \lambda- ||\br|| |\bx_j||
\frac{\lambda_0-\lambda}{\lambda_0}
\label{cond2}
\end{eqnarray}
We have been unable to prove the correctness of this rule,
and do not know if it is infallible.
At the same time,
we have been not been able to find a numerical example in which it fails.

\medskip

{\bf Acknowledgements:} 
We thank Stephen Boyd for his comments,
and Laurent El Ghaoui and his
co-authors for sharing their paper with us before publication,
and for helpful feedback on their work.
The first author was supported by National
Science Foundation Grant DMS-9971405 and National Institutes of Health
Contract N01-HV-28183.
\bibliographystyle{agsm}
\bibliography{strong,/home/hastie/bibtex/tibs,/home/hastie/docs/mrc/bibtex/mrc.bib,/home/tibs/texlib/tibs,/home/hastie/docs/resume/trevor}

\newpage
\begin{table}
\caption[tab1]{\em Glmnet timings (seconds) for  fitting a lasso problem in the Gaussian setting.
In the first four columns, there are $p=100,000$ predictors, $N=200$ observations, 30 nonzero coefficients, with the same value and  signs alternating; signal-to-noise ratio equal to 3.
In the rightmost column, the data matrix is sparse, consisting of just zeros and ones, with $0.1\%$ of the values equal to 1.
There are $p=50,000$ predictors, $N=500$ observations, with 25\% of the  coefficients nonzero,
having a Gaussian distribution; signal-to-noise ratio equal to 4.3.}
\label{tab1}
\centering
\fbox{%
\begin{tabular}{l|rrrrr}
Method& \multicolumn{4}{c}{Population correlation}\\ 
&                         0.0  &               0.25  &             0.5  &             0.75  & Sparse\\
\hline
glmnet&                   4.07        &      6.13      &        9.50       &     17.70 & 4.14\\
with seq-strong &       2.50  &               2.54     &        2.62       &       2.98 &2.52\\
\end{tabular}}
\end{table}

\begin{table}
\caption[tab3]{\em  Glmnet timings (seconds) for  fitting an elastic net  problem.
There are $p=100,000$ predictors, $N=200$ observations, 30 nonzero coefficients, with the same value and  signs alternating; signal-to-noise ratio equal to 3 }
\label{tab3}
\centering
\fbox{%
\begin{tabular}{l|rrrrr}
Method& \multicolumn{5}{c}{$\alpha$}\\
&                         1.0  &               0.5  &             0.2  &      0.1  &     0.01\\

\hline
glmnet&                      9.49        &       7.98  &            5.88  &          5.34    &        5.26\\
with seq-strong &         2.64           &    2.65    &          2.73     &       2.99  &          5.44\\
\end{tabular}}
\end{table}

\begin{table}
\caption[tab4]{\em  Glmnet timings (seconds)  fitting a lasso/logistic regression problem.
Here the data matrix is sparse, consisting of just zeros and ones, with $0.1\%$ of the values equal to 1.
There are $p=50,000$ predictors, $N=800$ observations, with 30\% of the  coefficients nonzero,
with the same value and  signs alternating; Bayes error equal to 3\%.}
\label{tab4}
\fbox{%
\centering
\begin{tabular}{l|rrrr}
Method& \multicolumn{4}{c}{Population correlation}\\
&                         0.0  &                 0.5  &             0.8\\
\hline
glmnet            &      11.71      &      12.41 &     12.69\\ 
with seq-strong    &     6.31      &      9.491 &     12.86\\
\end{tabular}}
\end{table}

\end{document}